\documentclass[final,onefignum,onetabnum]{siamonline171218}

\usepackage{mathtools, amssymb}
\usepackage{subcaption}
\usepackage[font=small,labelfont=bf]{caption}
\usepackage{hyperref}
\usepackage{tikz-cd}
\usepackage{tikz}
\usepackage{algorithm}
\usepackage{algorithmic}
\usepackage{eqparbox}

  \usepackage{nicefrac}
\usepackage[abs]{overpic}
\usepackage{etoolbox}  
\makeatletter

\patchcmd{\algorithmic}{\addtolength{\ALC@tlm}{\leftmargin} }{\addtolength{\ALC@tlm}{\leftmargin}}{}{}
\makeatother

\usepackage[mathscr]{euscript}
\usepackage{xcolor}

\DeclareMathOperator{\BOX}{box}
\newcommand{\paren}[1]{\left( #1 \right)}
\newcommand{\RHbox}[2]{\BOX\!\paren{#1;#2}}

\newcommand{\ceil}[1]{\left\lceil #1 \right\rceil}
\newcommand{\floor}[1]{\left\lfloor #1 \right\rfloor}
\newcommand{\Z}{\mathbb{Z}}
\newcommand{\R}{\mathbb{R}}
\newcommand{\set}[1]{\left\{ #1 \right\}}
\DeclareMathOperator{\relhaus}{\mathcal{RH}} 
\DeclareMathOperator{\relhausR}{\overrightarrow{\mathcal{RH}}} 
\newcommand{\RH}[1]{\relhaus\!\paren{#1}} 
\newcommand{\RHr}[1]{\relhausR\!\paren{#1}} 
\newcommand{\RHd}[1]{\relhaus_d\!\paren{#1}} 
\newcommand{\RHrd}[1]{\relhausR_d\!\paren{#1}} 
\newcommand{\abs}[1]{\left| #1 \right|}

\newcommand{\bigOh}[1]{\mathcal{O}\!\paren{#1}}

\definecolor{salmon}{RGB}{250,128,114}

\usepackage{lipsum}
\usepackage{amsfonts}
\usepackage{graphicx}
\usepackage{epstopdf}
\usepackage{algorithmic}
\ifpdf
  \DeclareGraphicsExtensions{.eps,.pdf,.png,.jpg}
\else
  \DeclareGraphicsExtensions{.eps}
\fi

\usepackage{enumitem}
\setlist[enumerate]{leftmargin=.5in}
\setlist[itemize]{leftmargin=.5in}

\newsiamremark{remark}{Remark}
\newsiamremark{hypothesis}{Hypothesis}
\crefname{hypothesis}{Hypothesis}{Hypotheses}
\newsiamthm{claim}{Claim}

\headers{A linear-time algorithm and analysis of graph Relative Hausdorff distance}{S.G. Aksoy, K.E. Nowak, and S.J. Young}

\title{A linear-time algorithm and analysis of graph Relative Hausdorff distance}

\author{Sinan G. Aksoy\thanks{Pacific Northwest National Laboratory, Richland, WA, 99352 
  (\email{sinan.aksoy@pnnl.gov}, \email{kathleen.nowak@pnnl.gov}, \email{stephen.young@pnnl.gov}).}
\and Kathleen E. Nowak\footnotemark[2]
\and Stephen J.\ Young\footnotemark[2]}

\usepackage{amsopn}

\ifpdf
\hypersetup{
  pdftitle={RH Paper},
  pdfauthor={S.G. Aksoy, K.E. Nowak, S.J. Young}
}
\fi

\begin{document}

\maketitle

\begin{abstract}
Graph similarity metrics serve far-ranging purposes across many domains in data science. As graph datasets grow in size, scientists need comparative tools that capture meaningful differences, yet are lightweight and scalable. Graph Relative Hausdorff (RH) distance is a promising, recently proposed measure for quantifying degree distribution similarity. In spite of recent interest in RH distance, little is known about its properties. Here, we conduct an algorithmic and analytic study of RH distance. In particular, we provide the first linear-time algorithm for computing RH distance, analyze examples of RH distance between pairs of real-world networks as well as structured families of graphs, and prove several analytic results concerning the range, density, and extremal behavior of RH distance values.
\end{abstract}



\section{Introduction and Background}

The need to quantify the similarity between graphs arises in domains ranging from graph sampling techniques \cite{Lee2006, Leskovec2006, Stumpf2005}, generative model selection \cite{Janssen2012, Motallebi2013}, botnet discovery \cite{Yin2013}, anomaly detection \cite{lafond2014anomaly, Ranshous2015}, classification and clustering of graph ensembles \cite{Airoldi2011, F.Costa2007}, and epidemic spreading \cite{Montanari2010, Pastor-Satorras2002}. In choosing a similarity function, the data scientist has no dearth of options at her disposal. Graph edit distance \cite{Sanfeliu1983}, maximum common subgraph based distance \cite{Fernandez2001}, iterative vertex-neighborhood methods \cite{Blondel2004, Kleinberg1999}, or indeed the difference in any selection of structural metrics like graph diameter, clustering coefficients, modularity, centrality, or eigenvalues, could potentially serve as similarity measures. However, as graph data becomes more massive, varied, and dense, there is an increasing need for comparison tools that are not only informative, but {\it lightweight} and {\it scalable}. In this regard, a number of tools unfortunately fall short. For example, while the aforementioned edit distance provides a nuanced and intuitive measure of similarity, its exact computation is NP-complete and even its approximation is APX-hard \cite{Lin1994}. 

One quantity that is easy to extract from data, and often used when comparing graphs, is the degree distribution. Perhaps the most fundamental and oft-studied notion in network science, the degree distribution plays an important role in graph modeling, algorithms, and resilience \cite{Barabasi1999, Cohen2000, Newman2003}. Due partly to its tractability and heavy-tailed nature, the degree distribution is particularly important for analyzing massive graphs like the Internet \cite{Broder2000, Faloutsos1999a}. Notwithstanding its ubiquity and importance, the degree distribution has limitations in capturing network structure. For example, it is straightforward to construct ensembles of graphs with identical degree distributions but strikingly different structural properties \cite{Doyle2005}. Nonetheless, because of its practicality, accessibility, and interpretability, the degree distribution remains a key representation of interest in graph data analysis.

Despite the apparent popularity of degree distributions, there is little consensus on how to meaningfully quantify their similarity. Although elegant and succinct, closed-form characterizations like the famous power-law are prone to statistical error when fitted to empirical data \cite{Clauset2009}.  A number of general-purpose statistical tools for probability distributions, such as the Kolmogorov-Smirnov (KS) test statistic, Earth movers distance, or percentile methods \cite{Janssen2012}, avoid imposing assumptions on degree distribution shape but, nonetheless, lack stringency in the graph setting \cite{Aliakbary2015, Matulef2017, Simpson2015}. Last but not least, researchers have proposed feature extraction methods which, based on the standard deviation of the degree distribution, extract a fixed-length feature vector from the degree distribution which is then used to define a distance function between distributions \cite{Aliakbary2015}. While a single, definitive metric for degree distribution comparison is clearly unrealistic, we believe a detailed, analytic study of the metrics themselves provides utility and garners insight into what the different metrics capture.

In this work, we conduct and algorithmic and analytical study of a newly proposed, sophisticated measure of degree distribution similarity called {\it Relative Hausdorff (RH) distance}. Graph RH distance was first introduced by Simpson, Seshadhri and McGregor \cite{Simpson2015} in the context of {\it small-space streaming algorithms}. 
Given a sequence or stream of edges $e_1, \dots, e_m$ of a graph, they devised an algorithm to estimate the graph's degree distribution when only a small fraction (typically less than 1\%) of the edge stream can be stored in memory at a time before being forgotten. 
In order to score the performance of this algorithm, they sought a similarity measure by which to quantitatively compare two graph degree distributions. Unsatisfied with the flaws in existing metrics commonly used for comparing degree distributions, Simpson et al. propose a new measure inspired by the Hausdorff metric from topology \cite{Hausdorff1914}, aimed at capturing degree distribution closeness at all scales.


In the context of degree distribution comparison, the purported benefits of RH distance over other metrics stem from a combination of factors. First, RH distance is flexible in tolerating variable error in both vertex degrees and their relative frequencies. This is especially useful for comparing heavy-tailed distributions typical of real graphs: for low degrees (when frequencies are large), RH distance allows for some error in frequency but less so in degree values whereas for high degrees (when frequencies are small), RH distance allows more error in degree estimates but less in frequency. Second, while flexible in this regard, RH distance is also stringent in requiring closeness at {\it every} point of the head, tail, and middle of degree distributions. Accordingly, any outlier behavior in one distribution must be approximated in the other. In contrast, a number of existing metrics like KS distance are insensitive to differences in the tails of degree distributions \cite{Simpson2015}. For these reasons, RH distance appears a promising tool for graph degree distribution comparison. 
Yet despite a spate of subsequent work \cite{Aksoy2018, Eden2018, Matulef2017, Stolman2017} applying RH distance to graph data, as well as recent work by the authors analyzing RH distance as an anomaly-detection method in cyber-security applications \cite{rhCyber}, little is known about the fundamental properties of this metric.


%

We organize our work as follows: in Section \ref{sec:prelims}, we define the original RH metric, as well as related variants from the literature. In Section \ref{sec:alg}, improving upon a previously known quadratic algorithm, we propose a new algorithm for computing RH distance and show its time-complexity is linear in the maximum degrees of the graphs being compared. This algorithm is not only of practical interest, but also provides a blueprint for analytically determining the RH distance between families of graphs. In Section \ref{sec:exs}, we obtain closed-form solutions for the RH distance between common families of graphs; these examples demonstrate that, contrary to its description in the literature as a ``distance metric," RH distance does not satisfy the triangle inequality. In Section \ref{sec:analysis} we analyze the possible values one can attain when computing the RH distance between two graphs. In particular, we (1) determine the {\it range} of possible RH values by establishing sharp upper bounds on the maximum RH distance between graphs; (2) show that the RH distance is {\it dense} in the real numbers; and (3) characterize the {\it extremal} behavior of the RH distance between graphs on a fixed vertex set that differ by a single edge. Throughout, we compare the theoretical results against those obtained by calculating the RH distance on pairs of real-world graphs from the Stanford Network Analysis Project (SNAP) \cite{leskovec2012stanford}.  This collection contains 40 graphs\footnote{All datasets were processed as simple graphs, ignoring any multiple, weighted, directed or self-loop edges.} with a variety of different structural properties and sizes. Finally, in Section \ref{sec:conc} we outline open questions on RH distance, and briefly suggest avenues for future work.



\section{Preliminaries} \label{sec:prelims}

The Relative Hausdorff distance between graphs is based on their complementary cumulative degree histograms (ccdh), which summarizes their vertex degree counts. As the ccdh is closely related to the oft-mentioned degree sequence and degree distribution of graphs, we define all three below for clarity. 


\begin{definition}
For an $n$-vertex graph $G=(V,E)$ with no isolated vertices, let $d(v)$ denote the degree of a vertex $v$. 
\begin{itemize}
\item[$(a)$] The degree sequence of $G$ is $\paren{d(v_i)}_{i=1}^{n}$.
\item[$(b)$] The degree distribution (i.e. degree histogram) of $G$ is $\paren{n(k)}_{k=1}^{\infty}$, where $n(k)=|\{v \in V : d(v)=k\}|$.
\item[$(c)$] The complementary cumulative degree histogram (ccdh) of $G$ is $\paren{N(k)}_{k=1}^{\infty}$, where $N(k)=\sum_{i=k}^\infty n(i)$.\end{itemize}
\end{definition}

Up to vertex labeling, all three of the above notions are equivalent in the sense that, for a given graph, each can be uniquely obtained from any of the others. However, for ease of exposition, it is often preferable to summarize degree counts using the ccdh, which (as noted in \cite{Simpson2015}) is typically less noisy than the degree distribution and additionally has the nice property of being monotonically decreasing. When clear from context, we will slightly abuse notation and write $F(d)$ for a graph $F$ to mean the value of the ccdh of $F$ at $d$. Furthermore, we let $\Delta(F)$ denote the maximum degree of graph $F$. Below, we present the original definition of RH distance, using slightly different notation than in \cite{Simpson2015}. 


\begin{definition}[Discrete RH \cite{Simpson2015}] \label{def:origRH}Let $F,G$ be graphs. The {\it discrete directional Relative Hausdorff distance} from $F$ to $G$, denoted $\RHrd{F,G}$, is the minimum $\epsilon$ such that
\[
\forall d \in \{1,\dots, \Delta(F)\}, \exists d' \in \{1, \dots, \Delta(G)+1\} \mbox{ such that } |d-d'| \leq \epsilon d \mbox{ and } |F(d)-G(d')|\leq \epsilon F(d). 
\]
We call $\RHd{F,G}=\max\{\RHrd{F,G},\RHrd{G,F}\}$ the {\it discrete Relative Hausdorff distance} between $F$ and $G$. 
\end{definition}


By definition, $\RHd{F,G}=\epsilon$ means that for every degree $k$ in the graph $F$, $F(k)$ is within $\epsilon$-fractional error of some $G(k')$, where $k'$ is within $\epsilon$-fractional error of $k$. In this sense, the RH measure is flexible in accommodating some error in both vertex degree values as well as their respective counts, yet strict in requiring that {\it every point} in $F$ be $\epsilon$-close to $G$ (and vice versa). Crucially, we note that the ccdh underlying Definition \ref{def:origRH} is defined on a discrete domain.

In \cite{Matulef2017, Stolman2017}, Matulef and Stolman argue that the discrete domain of the ccdh in the above definition of RH distance poses a potential flaw. In particular, they note that this may lead to counterintuitively large RH distances between pairs of graphs which are intuitively similar. Indeed, they give an example\footnote{Let $G$ consist of the complete bipartite graph $K_{2,n-2}$ and $G'$ be the same graph with one edge removed. By examining their respective ccdhs, it is clear that $|G'(n)-G(d)|\geq 1$ for every $d$; hence RH distance as given in Definition \ref{def:origRH} is 1.} of two graphs on the same vertex set which differ only by a single edge, yet still have (very large) RH distance of 1, for all values of $n$. To ameliorate this issue, they propose using a {\it smooth ccdh} in which each successive pair of points is connected via line segments. More formally, we define the smooth ccdh as follows.

\begin{definition} For a graph $G=(V,E)$, let $f(k)=|\{v: d(v) \geq k \}|$. The {\it smooth ccdh} of a graph $G$ is a function $\varphi: \mathbb{Z}_{\geq 1} \times [0,1) \to \mathbb{R}_{\geq 0}$ defined by 
\[
\varphi(x,\epsilon)=(1-\epsilon)f(x)+\epsilon f(x+1).
\] 
For ease of notation, we may write $\varphi(x) \coloneqq \varphi(\lfloor x \rfloor, x-\lfloor x \rfloor)$.
\end{definition}

Using smooth ccdhs, the {\it smooth RH distance} is defined analogously as in Definition \ref{def:origRH}, except $d'$ can now assume non-integer values. 

\begin{definition}[Smooth RH \cite{Matulef2017, Stolman2017}] \label{def:continuousRH}Let $F,G$ be graphs. The {\it smooth directional Relative Hausdorff distance} from $F$ to $G$, denoted $\RHr{F,G}$, is the minimum $\epsilon$ such that
\[
\forall d \in \{1,\dots, \Delta(F)\}, \exists d' \in [1, \Delta(G)+1]\mbox{ such that } |d-d'| \leq \epsilon d \mbox{ and } |\varphi_F(d)-\varphi_G(d')|\leq \epsilon \varphi_F(d). 
\]
We call $\RH{F,G}=\max\{\RHr{F,G},\RHr{G,F}\}$ the {\it smooth Relative Hausdorff distance} between $F$ and $G$.
\end{definition}

Comparing Definitions \ref{def:origRH}-\ref{def:continuousRH}, it is clear that for any pair of graphs, their smooth RH distance is at most their discrete RH distance. Henceforth, we focus exclusively on smooth RH distance and thus will abuse notation and use $F(x)$ for $\varphi_F(x)$.  

%

\section{A linear-time algorithm for RH Distance} \label{sec:alg}

The authors in \cite{Matulef2017, Simpson2015, Stolman2017} do not provide a pseudo-code specification of an algorithm for computing (discrete) RH distance between graphs. However, to supplement \cite{Stolman2017}, Python code implementing a straightforward, quadratic-time algorithm for computing smooth RH distance was made publicly available \cite{StolmanSoft2017}. Here, we provide a linear-time algorithm (\textsc{SmoothRH}) for computing RH distance in graphs. As an aside, we note that a linear-time algorithm for computing (non-graph) Hausdorff distance between convex polygons was given in \cite{Atallah1983}, although the same techniques do not apply in the case of graph Relative Hausdorff distance.

\begin{figure}[t!]
\begin{subfigure}[b]{0.5\textwidth}
\[
\scalebox{0.7}{
\begin{tikzpicture}
\draw (0,-1) -- (0,5);
\draw (-1,0) -- (7,0);
\node[below] at (3.5,0) {Degree $k$};
\node[below, rotate=90] at (-0.8,2.7) {$N(k)$};
\node at (4,2) [draw,circle,fill=black,inner sep = 0pt,minimum size =
5pt,label=45:{$(x,y)$}] {};
\node at (6,3) [draw,circle,fill=black,inner sep = 0pt,minimum size =
5pt,label=45:{$\mathbf{r}$}] {};
\node at (2,1) [draw,circle,fill=black,inner sep = 0pt,minimum size =
5pt,label=-135:{$\mathbf{\ell}$}] {};
\draw (2,1) -- (2,3) -- (6,3) -- (6,1) -- (2,1);
\draw[red,thick] (.5,4.75) -- (3,4.25) -- (4,4.25) -- (5,3.75) -- (6.5,3.5) --
(7,0);
\draw[blue,thick] (.5,4) -- (1.5,2) -- (2.25,2) -- (3,1.5) -- (3.76,1.5) -- (4,1.25) -- (5,.75) -- (6,0);
\draw[green,thick] (.5,2) -- (1,.5) -- (3,.5) -- (3.5,0);
\end{tikzpicture}
}
\]
\caption{} \label{F:box}
\end{subfigure}
\qquad
\begin{subfigure}[b]{0.4\textwidth}
\centering
\includegraphics[width=\linewidth]{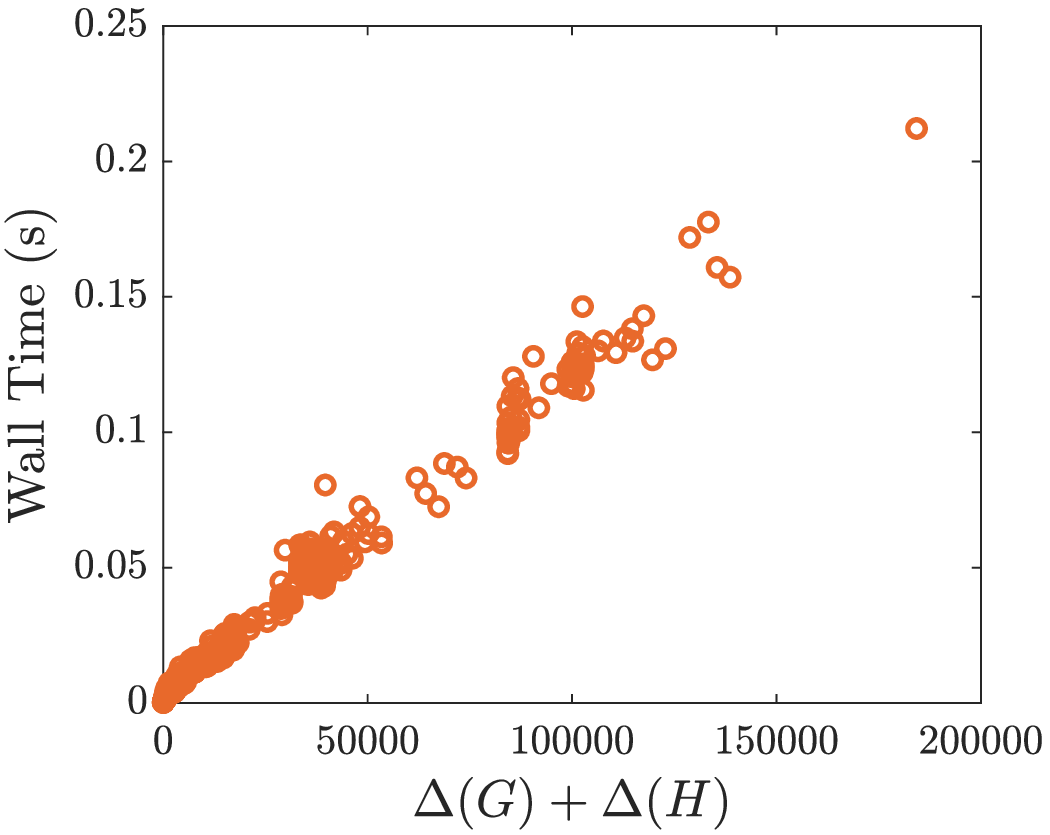}
\caption{} \label{walltime}
\end{subfigure}
\caption{An illustration of $\RHbox{(x,y)}{\epsilon}$ together with ccdh's passing above
  (red), through (blue), and below (green) the box (\ref{F:box}). Wall times for RH distance computation on the SNAP data (\ref{walltime}). }
\end{figure}

\tracingmacros=1
\begin{algorithm}
\caption{Maximum of  $\epsilon$ and smooth RH distance from $(x,y)$ to $H$}
\begin{algorithmic}[1]
\REQUIRE \textsc{SmoothRHdist}$((x,y),H,\epsilon)$
\STATE $(\ell_x,\ell_y) \gets ((1-\epsilon)x,(1-\epsilon)y)$  \COMMENT{lower left corner of $\RHbox{(x,y)}{\epsilon}$}
\STATE $(r_x,r_y) \gets ((1+\epsilon)x,(1+\epsilon)y)$\COMMENT{upper right corner of $\RHbox{(x,y)}{\epsilon}$}
\IF[$H$ passes well above $\RHbox{(x,y)}{\epsilon}$]{$\ell_x < 1$\ \AND \ $H(1) < \ell_y$} \label{alg:cond1}
\RETURN $1 - \frac{H(1)}{y}$
\ELSIF[{$H$ passes below $\RHbox{(x,y)}{\epsilon}$}]{$\ell_x \geq 1$ \ \AND \ $H(\ell_x) < \ell_y $}\label{alg:cond2}
\STATE $j \gets \ceil{\ell_x}$
\WHILE{$H(j-1) < \frac{(j-1)y}{x} \ \AND \ j > 1$}\label{alg:left}
 \STATE $j \gets j-1$ 
\ENDWHILE
\IF{$j = 1$} \label{alg:check}
\RETURN $1 - \frac{H(1)}{y}$
\ELSE
\RETURN $\frac{y + (x-j)H(j-1) -  (x-j+1)H(j)}{y + x\paren{H(j-1) - H(j)}}$ 
\ENDIF
\ELSIF[{$H$ passes above $\RHbox{(x,y)}{\epsilon}$}]{ $r_x < \Delta(H)+1 \ \AND \ H(r_x) > r_y $} \label{alg:cond3}
\STATE $j \gets \floor{r_x}$
\WHILE{$H(j+1) > \frac{(j+1)y}{x}$ }\label{alg:right}
\STATE $j \gets j+1$
\ENDWHILE
\RETURN $\frac{(j+1-x)H(j) + (x-j)H(j+1) -y}{y +x\paren{H(j) - H(j+1)}}$
\ELSE[$H$ passes through $\RHbox{(x,y)}{\epsilon}$] 
\RETURN $\epsilon$
\ENDIF
\end{algorithmic}
\end{algorithm}


\begin{algorithm}
\caption{Smoothed Relative Hausdorff distance between the ccdhs $F$
  and $G$}
\begin{algorithmic}[1]
\REQUIRE \textsc{SmoothRH}$(F,G)$
\STATE $\epsilon \gets 0$
\STATE $\Delta \gets \max\{\Delta(G),\Delta(F)\}$ \COMMENT{maximum degree occurring in $F$ or $G$}
\FOR{$i \in 1,\ldots,\Delta$}
    \IF{$i \leq \Delta(F)$}
    \STATE $\epsilon \gets \textsc{SmoothRHdist}((i,F(i)),G,\epsilon)$ \label{alg:G}
    \ENDIF
\IF{$i \leq \Delta(G)$}
\STATE $\epsilon \gets \textsc{SmoothRHdist}(i,G(i)),F,\epsilon)$ \label{alg:F} 
\ENDIF
\ENDFOR
\RETURN $\epsilon$
\end{algorithmic}
\end{algorithm}

For convenience of notation, for any graph $H$ we will interpret the smooth ccdh as taking value 0 for any $x > \Delta(H)+1$ even though this region is not in the domain. Before beginning the analysis of the \textsc{SmoothRHdist} and
 \textsc{SmoothRH} algorithms, we first introduce notation for the
relative Hausdorff equivalent of a ball, 
   \[\RHbox{(x,y)}{\epsilon} =
     \set{ (s,t) \in \R^2 \colon |x-s|  \leq \epsilon x, |y -t| \leq \epsilon y}.\]
    Thus, $\RHr{F,G}$ may be thought of as the least $\epsilon$ such that for all $1 \leq d \leq \Delta(F)$
     the smooth ccdh for $G$ intersects $\RHbox{(d,F(d))}{\epsilon}$.  We note that if $\epsilon < \epsilon'$, then $\RHbox{(x,y)}{\epsilon} \subset \RHbox{(x,y)}{\epsilon'}$ and hence the correctness of \textsc{SmoothRH} follows immediately from the correctness of \textsc{SmoothRHdist}.

To establish the correctness of \textsc{SmoothRHdist}, we first note that as the smooth ccdh's are monotonically decreasing it is easy to determine whether the smooth ccdh for $H$ intersects $\RHbox{(x,y)}{\epsilon}$ for some fixed $(x,y)$ and $\epsilon$.  Specifically,  the smooth ccdh of $H$ intersects $\RHbox{(x,y)}{\epsilon}$ if and only if either:
     \begin{itemize}
     \item $(1-\epsilon)x < 1$ and $H(1) \geq (1-\epsilon)y$, or
     \item the ccdh of $H$ is defined at $(1-\epsilon)x$ and is greater than $(1-\epsilon)y$ and is undefined at $(1+\epsilon)x$, or
     \item the ccdh of $H$ is defined at $(1-\epsilon)x$ and $(1+\epsilon)x$ and is greater than $(1-\epsilon)y$ at $(1-\epsilon)x$ and is less than $(1+\epsilon)y$ at $(1+\epsilon)x$.
     \end{itemize}
     It is straightforward to verify that the combined negation of the conditions on lines \ref{alg:cond1}, \ref{alg:cond2}, and \ref{alg:cond3} results in one of these conditions holding, thus $\RHbox{(x,y)}{\epsilon}$ intersects the smooth ccdh of $H$ and \textsc{SmoothRHdist} correctly returns $\epsilon$.

     We now consider the situation when the smooth ccdh of $H$ does not intersect any portion of $\RHbox{(x,y)}{\epsilon}$.  Broadly speaking, this situation can be defined into two cases; the ccdh is above the box or the ccdh is below the box. These scenarios are illustrated by the red (above the box) and green (below the box) lines in Figure \ref{F:box}.

    In the case where $\RHbox{(x,y)}{\epsilon}$ is below the smooth ccdh of $H$, the return value of $\textsc{SmoothRHdist}$ should be the least $\epsilon'$ such that $\mathbf{r} = \paren{(1+\epsilon')x,(1+\epsilon')y}$ is on the smooth ccdh of $H$.  Noting that $\paren{(1+\epsilon')x,(1+\epsilon')y}$ implicitly defines a line $L$ and the smoothed ccdh is piecewise linear, we can find such an $\epsilon'$ by identifying a pair $(j,j+1)$ where $L$ and the smooth ccdh for $H$ cross.  This achieved by an iterative search over such pairs starting on line \ref{alg:right} and returning $\epsilon'$ which determines the unique point of intersection between $L$ and the smooth ccdh of $H$. 

For the case where $\RHbox{(x,y)}{\epsilon}$ is above the smooth ccdh of $H$, there are essentially two cases; either  
    the line $L$ implicitly defined by $\paren{(1-\epsilon')x,(1-\epsilon')y}$ passes through the smooth ccdh of $H$ or the desired $\epsilon'$ is given by when $(1-\epsilon')y = H(1)$. Similarly as to the case where $\RHbox{(x,y)}{\epsilon}$ is below the smooth ccdh of $H$, the intersection point with $L$ is determined by iteratively searching to the left. The possibility that $L$ doesn't intersect with the smooth ccdh handled by the cases on line \ref{alg:cond1} and \ref{alg:check}.  Together these three cases yield the following result. 
\begin{lemma}
Algorithm \textsc{SmoothRHdist} correctly calculates the maximum of
$\epsilon$ and the relative
Hausdorff distance between the point $(x,y)$ and the complimentary
cumulative degree distribution $H$.
\end{lemma}

We now turn to the overall run time of the algorithm \textsc{SmoothRH}.

\begin{theorem} \label{thm:lineartime}
The algorithm \textsc{SmoothRH} calculates the smooth Relative
Hausdorff distance between two complimentary cumulative degree
distributions $F$ and $G$ in time at most $\Theta\paren{\Delta(G) + \Delta(F)}$.
\end{theorem}

\begin{proof}
We note that a cursory examination of \textsc{SmoothRH} and
\textsc{SmoothRHdist} reveals that running time of \textsc{SmoothRHdist} 
is driven by the number of times the while loops at Line
\ref{alg:left} and \ref{alg:right} are executed.  It is easy to see
that this gives an $\bigOh{\Delta(F)^2+\Delta(G)^2}$ upper bound on the
run time.  To show that the run time is linear we will in fact show that these while loops are entered at most
$\Delta(F) + \Delta(G)$ times giving a overall run time of
$\bigOh{\Delta(F) + \Delta(G)}$.  We note that by the symmetry of the process, it suffices to show that the total amount of work over all calls of the form $\textsc{SmoothRHdist}\paren{(i,F(i)),G,\epsilon}$ is $\bigOh{\Delta(F)+\Delta(G)}$.  For clarity of notation, we will denote by $\epsilon_i$ the value of $\epsilon$ passed to \textsc{SmoothRH} when processing the pair $(i,F(i))$.

    It is relatively easy to see that computational cost of the while loop at line \ref{alg:right} over all $i$ is at most $\bigOh{\Delta_F + \Delta_G}$ by the monotonicity of the right edge $\RHbox{(i,F(i))}{\epsilon_i}$.  Specifically, as $\epsilon_i$ is a monotonically increasing function of $i$, so is $r_i = \floor{ (1+\epsilon_i)i}$.  But then, other than a constant amount of work associated with each iteration $i$, this implies that for every $1 \leq k \leq \Delta(G)+1$, the value of $G(k)$ is retrieved at most once.  Thus the total accesses to the ccdh of $G$ are bounded by $\bigOh{\Delta(F) + \Delta(G)}$ as desired.

    Now we consider the case where $\RHbox{(i,F(i))}{\epsilon_i}$ is above the smooth ccdh for $G$.  We note that in this case $\ell_i = \ceil{(1-\epsilon_i)i}$ is not a monotonic function of $i$, so we will have to apply different arguments.  To that end, let $j = \ceil{(1-\epsilon_{i+1})i}$, i.e. the value of $j$ which terminates the while loop on line \ref{alg:left}.  Suppose by way of contradiction that, $1 < j < i$.  First we note that,  $\RHbox{(j,G(j))}{\epsilon_{j+1}}$, $\RHbox{(j-1,G(j-1))}{\epsilon_j}$ both contain the smooth ccdh for $F$ and furthermore, the smooth ccdh of $F$ 
      must lie in the union of halfspaces
      \[\mathcal{U} = \set{(x,y) \colon y \geq F(i)} \cup \set{(x,y) \colon y \geq -F(i)(x-i) + F(i)}.\]
      Further, $j-1< (1-\epsilon_{i+1})i$, we have that
      \[ (1+\epsilon_j)(j-1) \leq (1+\epsilon_i)(j-1)< (1+\epsilon_i)(1-\epsilon_{i+1})i < (1+\epsilon_{i+1})(1-\epsilon_{i+1}) < i.\]
      Thus, in order for $\RHbox{(j-1,G(j-1))}{\epsilon_{j-1}}$ to contain the smooth ccdh of $F$, we have that
      \[ F(i) \leq (1+\epsilon_{j-1})G(j-1) \leq (1+\epsilon_i)G(j-1).\] In a similar manner, we have that
      \[ -F(i)\paren{(1+\epsilon_i)j - i} + F(i)\leq (1+\epsilon_i)G(j).\]
      Now letting $\delta = (1-\epsilon_{i+1})i -(j-1)$,  we have by the definition of $\epsilon_{i+1}$, that
      \[ (1-\epsilon_{i+1}) F(i) = G(j-1) + \delta\paren{G(j)-G(j-1)}.\] 
    Combining these inequalities we get that
      \begin{align*}
        \paren{1+\epsilon_i}\paren{1-\epsilon_{i+1}}F(i) &= \paren{1+\epsilon_i}(1-\delta)G(j-1) +\paren{1+\epsilon_i}\delta G(j) \\
                                                         &\geq \paren{1-\delta}F(i) + \delta\left[ -F(i)\paren{(1+\epsilon_i)j - i} + F(i) \right] \\
                                                         &= \left[ (1-\delta) - \delta(1+\epsilon_i)j +\delta i+\delta \right] F(i) \\
                                                         &= \left[ 1 + \delta i - \delta(1+\epsilon_i)j\right] F(i).
      \end{align*}
      As $F(i) > 0$ and $j = (1-\epsilon_{i+1})i + (1-\delta)$, this implies that
      \[ 
      (1+\epsilon_i)(1-\epsilon_{i+1}) \geq  1+ \delta\left[ i -  (1+\epsilon_i)(1-\epsilon_{i+1})i + (1-\delta)(1+\epsilon_i)\right]. 
      \]
      Since $\epsilon_i < \epsilon_{i+1}$, we have that $(1+\epsilon_i)(1-\epsilon_{i+1}) < 1$.  But this yields that the left hand side of the above inequality is strictly smaller than 1, while the right hand side is strictly larger than 1, a contradiction.
      
     Thus, the while loop on line \ref{alg:left} can only terminate when $j = i$ or $j=1$.  The second case represents only a constant amount of work per iteration, contributing at most $\Delta(F)$ to the overall runtime.  For the first, we notice that this occurs only if $G(1) < \frac{1}{i} F(i)$.  In this case, $\epsilon_{i+1} = 1-\frac{G(1)}{F(i)}$ and in particular for all $k \geq i+1$,
      \[ (1-\epsilon_k)F(k) \leq (1-\epsilon_{i+1}) F(i) = G(1).\]  But this implies that line \ref{alg:left} can terminate with $j = 1$ precisely once when called as a subroutine of \textsc{SmoothRH} and hence the total running time is $\Theta\paren{\Delta(F) + \Delta(G)}$, as desired. 
      
  \end{proof}

Lastly, we examine the time complexity of \textsc{SmoothRH} in practice. In Figure \ref{walltime}, we plot wall times for RH distance computation on 40 graph datasets collected from SNAP \cite{leskovec2012stanford}, yielding a total of ${40 \choose 2}=780$ pairs of graphs. This experiment was run on a 2015 Macbook Pro laptop with a 2.8 GHz Intel Core i7 processor using a Python implementation of Algorithm \textsc{SmoothRH}. Observe the linear relationship apparent in Figure \ref{walltime} is congruent with the time complexity stated in Theorem \ref{thm:lineartime}. In the next section, we present the RH distance values between these graphs. 

\section{Examples and the triangle inequality}\label{sec:exs}

\begin{figure}[t!]
\centering
\includegraphics[scale=0.45]{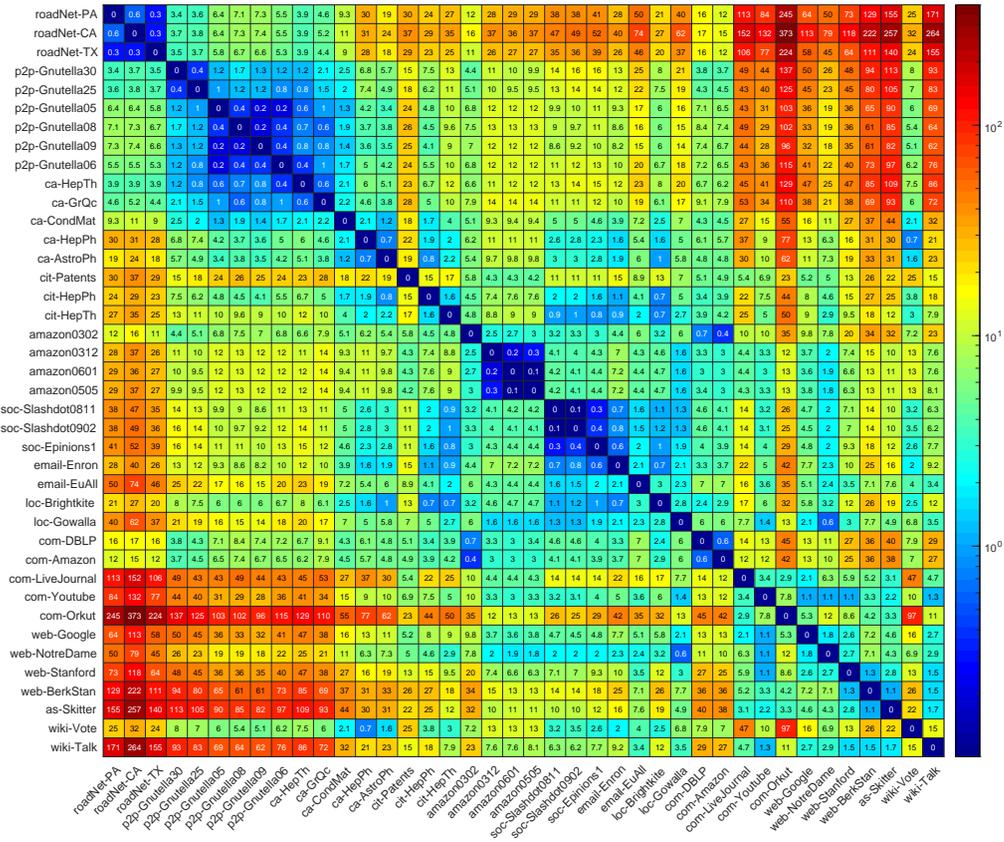}
\caption{Heat map of RH distance values between 40 graph datasets from SNAP \cite{leskovec2012stanford}.}\label{fig:heat}
\end{figure}

In introducing both discrete \cite{Simpson2015} and smooth \cite{Matulef2017} RH distance, the authors examine RH distance on simple families of graphs to illustrate its behavior and compare it with other degree-based comparison metrics. 
In this section, we present examples of RH distance between real graph datasets and analyze RH distance between structured families of graphs. 


Figure \ref{fig:heat} depicts a heat map of RH distance values between pairs of graphs collected from SNAP \cite{leskovec2012stanford}. The datasets are grouped by network type (as indicated by their SNAP label prefix), sorted within each group by maximum degree, and then sorted by group average maximum degree. The RH values range from 0.1 to 373, with mean 20 and median 8.49. The blue blocks on the diagonal identify network types with relatively small intra-group RH distance, such as road, peer-to-peer, and co-authorship networks. In contrast, citation and communication networks feature larger intra-group RH distance values. The red portions in the outer corners of the heat map indicate large RH distance between graphs with large differences in maximum degree. For instance, \texttt{roadNet-CA} and \texttt{com-Orkut} achieve the largest RH distance, and have maximum degree 12 and 33,313, respectively. This pattern reflects that, while not solely controlled by maximum degree, RH distance is sensitive to differences in the tails of degree distributions. 

Next, we turn our attention to analyzing RH distance between structured families of graphs. To this end, we first establish a lemma useful for analytically determining RH distance between graphs whose degree sequence satisfies certain properties. 
Applying this lemma, we then find closed-form expressions for RH distance between complete graphs of different sizes, and between the complete graph and cycle. Lastly, we conclude this section by analyzing these examples to show RH distance does not satisfy the triangle inequality, nor can the definition of RH distance be modified in certain ways to satisfy the triangle inequality.

\begin{lemma}\label{L:flat}
Let $F$ and $G$ be graphs and suppose that $F$ has no vertices of degree $k$ for $d \leq k \leq d'$.  Let $1 \leq x,x' \leq \Delta(G)+1$ and $\epsilon >0$ such that 
\begin{align*}
\abs{d - x} &\leq \epsilon d, & && \abs{d'+1-x'} &\leq \epsilon (d' +1), \\
\abs{F(d) - G(x)} &\leq \epsilon F(d), &  \textrm{ and }&& \abs{F(d'+1) - G(x')} &\leq \epsilon F(d'+1).
\end{align*}
Then for all $d < \delta <$ 
$d'+1$, $G$ intersects $\RHbox{(\delta, F(\delta))}{\epsilon}$.
\end{lemma}

\begin{proof}
We can choose $x$ and $x'$ such that $x \leq x'$ so assume they are chosen as such. Fix $d < \delta < d'+1$ and note that since there are not vertices of degree between $d$ and $d'$, we have that there is some $y = F(d) = F(\delta) = F(d'+1)$.  By assumption we have that $\RHbox{(d,y)}{\epsilon}$ and $\RHbox{(d'+1,y)}{\epsilon}$ both intersect the smooth ccdh of $G$.  Noting that 
$$
(1+\epsilon)F(\delta) \geq G(x) \geq G(x') \geq (1-\epsilon)F(\delta),
$$
$G$ is monotonically decreasing and $[x, x'] \cap [(1-\epsilon) \delta, (1+\epsilon)\delta] \neq \varnothing$, completes the proof.
\end{proof}

As a practical consequence of Lemma \ref{L:flat}, it suffices to consider the end points of the plateaus present in the smooth ccdh of $F$ when calculating $\RHr{F,G}.$  This observation makes it a straightforward exercise to verify the following result:
\begin{lemma}\label{L:complete}
Let $3 \leq n \leq m$, then $\RH{K_n,K_m} = \frac{m-n}{n}$.
\end{lemma}
\begin{proof}
By Lemma \ref{L:flat} it suffices to consider the RH distance from $(1,n)$ and $(n-1,n)$ to the ccdh of $K_m$ and from $(1,m)$ and $(m-1,m)$ to ccdh of $K_n$. We note that for any $\delta > 0$, the upper right corner of $\RHbox{(1,n)}{\frac{m-n}{n} - \delta}$ is $\paren{\frac{m}{n}-\delta, m - \delta n}$.  As $\frac{m}{n}-\delta < m-1$ and $m - \delta n < m$, this implies that $\RH{K_n,K_m} \geq \frac{m-n}{n}$.  To show that $\RH{K_n,K_m} \leq \frac{m-n}{n}$, we first note that $\paren{1 - \frac{m-n}{n}}m = 2m - \frac{m^2}{n} \leq n$, and so $\RHbox{(x,m)}{\frac{m-n}{n}}$ contains a point on the ccdh of $K_n$ as long as $(1-\frac{m-n}{n})x \leq n-1$.  This is obvious for $x =1$ and for $x = m-1$ the inequality $(2n-m)(m-1) \leq n(n-1)$ can be rearranged in to $m-n \leq \paren{m-n}^2$ which clearly holds.  Now we note that $\RHbox{(x,n)}{\frac{m-n}{n}}$ contains $(x,m)$ on the boundary and thus $\RH{K_n,K_m} = \frac{m-n}{n}$.  
\end{proof}

Along similar lines it is relatively straightforward, but tedious to derive closed-form expressions of RH distance between different families of graphs, such as the complete graph and cycle:
\begin{lemma}
Let $3 \leq n,m$, then
\[ \RH{K_n,C_m} = \begin{cases} \frac{n}{m} - 1 & 3 \leq m < \frac{n^2-3n + \sqrt{n^4+2n^3-11n^2}}{4n-10} \\
 1 - \frac{3m}{n + nm - m}&  \frac{n^2-3n + \sqrt{n^4+2n^3-11n^2}}{4n-10} \leq m \leq n\\
 1 - \frac{3m}{n + nm - m} & 5 \leq n < m \leq \frac{n^2 -3n + \sqrt{n^4-4n^3+7n^2}}{n-1}\\
 \frac{m}{n}-1&  \frac{n^2 -3n + \sqrt{n^4-4n^3+7n^2}}{n-1} < m \leq 2n \\
 \frac{3m}{n+m} - 1 & 2n <m  \\ 
\end{cases}. \]
\end{lemma}

In addition to illustrating the nuance of RH distance values even for simple examples like the complete graph and cycle, such closed-form expressions are also useful in revealing mathematical properties of RH distance. For instance, a cursory analysis of the closed-form expression in Lemma \ref{L:complete} reveals that, while the RH measure clearly satisfies the non-negativity and symmetry axioms defining a distance metric, RH distance does not satisfy the triangle inequality: observe that $\RH{K_3,K_4}+\RH{K_4,K_5}=7/12<2/3=\RH{K_3,K_5}$. A natural, subsequent question is whether RH distance satisfies a {\it $c$-relaxed triangle inequality}, a triangle inequality variant that has been utilized in pattern matching \cite{Fagin1998}, of the form
\[
\RH{X,Y}+\RH{Y,Z} \geq c \cdot \RH{X,Z},
\]
for some fixed $c\leq 1$. However, a closer analysis of the closed-form expression in Lemma \ref{L:complete} reveals that a $c$-relaxed triangle inequality is also impossible.

\begin{claim}[Relaxed triangle inequality impossible]
For any $\epsilon>0$, there exist graphs $F,G,H$ such that 
\[
\frac{\RH{F,G}+\RH{G,H}}{\RH{F,H}}<\epsilon.
\]
\end{claim}

\begin{proof}
We take $F,G$ and $H$ as the complete graph on $x,y$ and $cx$ vertices, respectively, for some constant $c> 1$ where $x<y<cx$. Using the closed-formula for the RH distance between complete graphs provided in Lemma \ref{L:complete}, the ratio in question is
\[
f(y) := \frac{1}{c-1}\left( \frac{y}{x} + \frac{cx}{y}-2 \right),
\]
and choosing $y=\sqrt{c}x$ which minimizes $f$, we see that $f(\sqrt{c}x) \to 0$ as $c \to \infty$. 
\end{proof}

Lastly, we consider the possibility that the definition of RH distance itself could be modified in some way so that it satisfies the triangle inequality. Recall that the RH distance between graphs $G$ and $F$ is the maximum of the two directional RH distances between $G$ and $F$. Instead of taking the extreme values of directional RH distance, one approach may be to  ``regularize" RH distance by taking some weighted average of directional RH distances. However, we provide a simple counterexample below to show no positive linear combination of maximum or minimum directional RH distances defines a bona-fide distance metric. 




\begin{claim} There exists no positive linear combination of $\min\{\RHr{F,G},\RHr{G,F} \}$ and $\max\{\RHr{F,G},\RHr{G,F} \}$ which defines a distance metric on graphs.\end{claim}

\begin{proof}
Consider the three trees on 6 vertices illustrated in Figure \ref{fig:trees}. A straightforward computation shows
\begin{alignat*}{2}
 \RHr{F,G} &=\nicefrac{1}{3} &\mbox{ and } \RHr{G,F} &= \nicefrac{1}{12}\\
\RHr{F,H} &=1 &\mbox{ and } \RHr{H,F} &= \nicefrac{1}{2} \\
\RHr{H,G} &=\nicefrac{1}{5} &\mbox{ and } \RHr{G,H}&=\nicefrac{1}{7}
\end{alignat*}
from which we observe that no positive linear combination of the maximum and minimum of the directional RH measures will always satisfy the triangle inequality since for any $a,b>0$, we have $2(a+b/2) > a(\frac{1}{3}+1+\frac{1}{5})+b(\frac{1}{12}+\frac{1}{2}+\frac{1}{7})$. 
\end{proof}

%

\begin{figure}[t]
\begin{subfigure}[b]{0.5\textwidth}
\[
\begin{tikzpicture}[sibling distance=4em, scale=0.4,
  every node/.style = {shape=circle, minimum size=1mm, inner sep=2pt,
    draw, align=center, fill=black}]]
  \node {}
    child { node {} 
    child {node {} }}
    child { node {} }
    child { node {} }
    child { node {} };
\end{tikzpicture}
\qquad \qquad
\begin{tikzpicture}[sibling distance=4em, scale=0.4,
  every node/.style = {shape=circle, minimum size=1mm, inner sep=2pt,
    draw, align=center, fill=black}]]
  \node {}
    child { node {} 
    child {node {} }
    child {node {}}}
    child { node {} }
    child { node {} };
\end{tikzpicture}
\]
\\
\[
\begin{tikzpicture}[sibling distance=4em, scale=0.4,
  every node/.style = {shape=circle, minimum size=1mm, inner sep=2pt,
    draw, align=center, fill=black}]]
  \node {}
    child { node {} }
    child {node {} }
    child { node {} }
        child { node {} }
    child { node {} };
\end{tikzpicture}
\vspace{6mm}
\]
\caption{}\label{fig:trees}
\end{subfigure}
\begin{subfigure}[b]{0.5\textwidth}
\centering
\includegraphics[width=0.8\linewidth]{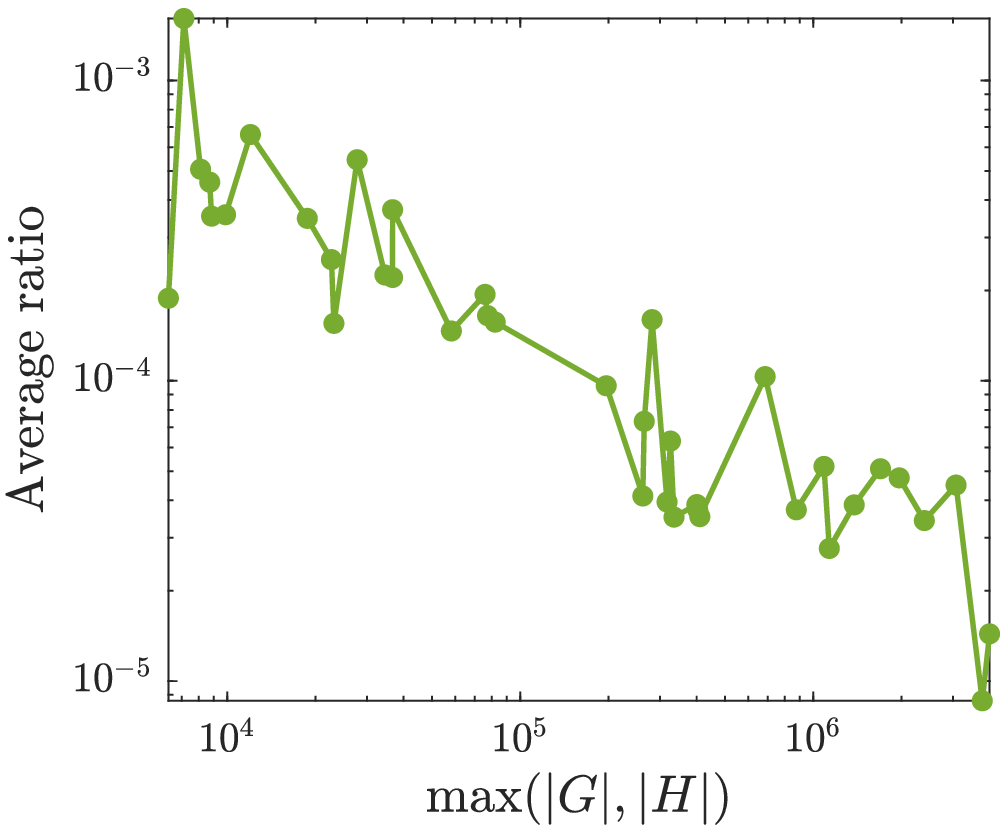}
\caption{}\label{ratio_plot}
\end{subfigure}
\caption{Trees $G$ (top left), $H$ (top right), and $F$ (bottom) (\ref{fig:trees}). The ratio of $\mathcal{RH}(G,H)$ to the max possible RH distance, averaged over all pairs $G,H$ in the SNAP data with $m=\max(|G|,|H|)$ (\ref{ratio_plot}). } 
\end{figure}


In short, the notion of RH distance has seemingly intrinsic barriers to satisfying the triangle inequality. We note that RH distance's status as a bona-fide distance metric may not solely be a question of theoretical interest since performing graph distance computations over objects embedded in a metric space can sometimes have practical benefits, such as faster graph algorithms \cite{Bento2018}. Next, we consider other analytic properties of RH distance useful for understanding its behavior. 

\section{Range, density, and extremal properties of RH distance} \label{sec:analysis}

In this section, we derive several analytic results on RH distance which characterize the range, density and extremal behavior of RH distance values. 
\subsection{Range}

A number of distance metrics, such as KS distance, are (by definition) bounded between 0 and 1. This normalization is a desirable property and facilitates interpretability when comparing distances between pairs of graphs of varying sizes. In the case of RH distance, while the Simpson et al.~\cite{Simpson2015} note ``RH distance can exceed 1", no upper bound is known. Below, we establish a sharp upper bound on RH distance. 

\begin{theorem}\label{fig:maxPoss}
Let $\mathcal{G}_k$ denote the set of all graphs on $k$ vertices. If $n \leq m$ then 
\[ \max_{\substack{F \in \mathcal{G}_n \\ G \in \mathcal{G}_m}}
    \RH{F,G} = \begin{cases} \frac{m}{2} - 1 & n= 2, \\
\frac{m^2}{2m+1} - 1 & n \geq 3,m \geq 4\end{cases}.\]
\end{theorem}

\begin{proof}
We first consider the case where $n=2$ and $m \geq 3$.  We note that $K_2$ is the unique graph on two vertices with no isolated vertices.  Thus since $\{(1,i) \hspace{1mm} | \hspace{1mm} i \in [2,m]\} \subset \RHbox{(1,2)}{\frac{m}{2}-1} $, we have that 
\[ \max_{\substack{F  \in \mathcal{G}_2 \\ G \in \mathcal{G}_m}} \RH{F,G} > \frac{m}{2}-1 \] if and only if there is some $G \in \mathcal{G}_m$ such that $\RHr{G,K_2}>\frac{m}{2}-1$.

Now we note that $\RHbox{(x,G(x))}{\frac{m}{2}-1}$ contains the point $(1,2)$ for all $x \in Z_{\geq 1}$ when $m \neq 3$. Thus, we are left to consider when $m = 3$. In this case, $\RHr{P_3, K_2} = \RHr{K_3,K_2} = \frac{1}{3}$. Thus, there is no $G \in \mathcal{G}_m$ such that $\RHr{G,K_2} > \frac{m}{2}-1$. Finally, it is easy to verify that $\RHr{K_2,K_m} = \frac{m}{2}-1$, which completes the case when $n = 2$.

For ease of notation, in the remainder of the proof we will define $\epsilon_m$ as $\frac{m^2}{2m+1}-1$.  Suppose now that $F$ has exactly one vertex of degree greater than $1$ and that $G = K_m$.  Then 
\[ \RHbox{(2,1)}{\epsilon_m} = \left[2-2\epsilon_m, \frac{2m^2}{2m+1}\right] \times \left[1-\epsilon_m,\frac{m^2}{2m+1}\right]\] which intersects the ccdh for $G$ precisely on the line between $(m-1,m)$ and $(m,0)$.  Thus $\RH{F,G} \geq \epsilon_m$.  Thus it suffices to show that for arbitrary choices of $F$ and $G$, $\RH{F,G} \leq \epsilon_m$.  

Now consider the case where $3 \leq n$ and $5 \leq m$.  Since $\epsilon_m > 1$ we have that for any $(x,y) \in \Z^2$, $\RHbox{(x,y)}{\epsilon_m}$ contains the set $[0,m] \times \set{y}$ if $x > 2$ and the set $\set{x} \times [0,m]$ if $y>2$.  
In particular, this implies that if $F \in \mathcal{G}_n$ and $G \in \mathcal{G}_m$ then $\RH{F,G} \leq \epsilon_m$ unless one of $F(2)$ or $G(2)$ is in $\set{1,2}$, as both $\RHbox{(1,G(1)}{\epsilon_m}$ and $\RHbox{(1,F(1))}{\epsilon_m}$ contain the points $(1,n)$ and $(1,m)$ since $3 \leq n,m$. We now consider the intersection of the final line segment in a ccdh and the boxes around $(2,1)$ and $(2,2)$.  To this end, let the endpoints of the final line segment be $(\Delta,c)$ and $(\Delta+1,0)$. Now we note that $[0,\frac{2m^2}{2m+1}] \times \set{0} \subset \RHbox{(2,1)}{\epsilon_m} \cap \RHbox{(2,2)}{\epsilon_m}$ and thus contains $(\Delta+1,0)$ unless $\Delta = m-1$.  It is then a simple algebraic exercise to show that the line segment $(m-1,c)$ to $(m,0)$ intersects $\RHbox{(2,1)}{\epsilon_m} \cap \RHbox{(2,2)}{\epsilon_m}$ for all $c \leq m$.  As $c$ is at most $m$ by definition, this completes the this case.  

The only remaining cases are when $m = 3, 4$.  Consider first when $n = m = 3$.  In this case, either the ccdhs are identical or one is $(3,1)$ and the other is $(3,3)$.  Given that the only integer point these differ on is $d =2$, it suffices to consider 
\[\RHbox{(2,1)}{\epsilon_3} = \RHbox{(2,1)}{\frac{2}{7}} = \left[\nicefrac{10}{7},\nicefrac{18}{7}\right] \times \left[\nicefrac{5}{7},\nicefrac{9}{7}\right].\] 
Noting that the point $\paren{\nicefrac{18}{7},\nicefrac{9}{7}}$ is on the line segment between $(2,3)$ and $(3,0)$ completes this case. 

To consider the case where $m = 4$ and $n = 3,4$, we first fix some $F \in \mathcal{G}_3 \cup \mathcal{G}_4$ and $G \in \mathcal{G}_4$ and suppose that $\RH{F,G} > \epsilon_4.$  Noting that the only valid ccdhs on 4 vertices are  $\set{(4), (4,2), (4,4), (4,1,1), (4,3,1),(4,4,2),(4,4,4)}$, it suffices to examine the boxes around a set of 9 points given in Table \ref{T:boxes}.  
\begin{table}
\centering
\begin{tabular}{c | c | c | c | c}
$(x,y)$ & $\RHbox{(x,y)}{\epsilon_4}$ & contains $(1,3)$ & contains $(1,4)$ & contains $[1,3]\times \set{1}$ \\ \hline \hline
$(1,3)$ & $\left[ \nicefrac{2}{9} , \nicefrac{16}{9} \right] \times \left[ \nicefrac{6}{9} , \nicefrac{48}{9} \right] $  & X & X & X\\ \hline
$(1,4)$ & $\left[ \nicefrac{2}{9} , \nicefrac{16}{9} \right] \times \left[ \nicefrac{6}{9} , \nicefrac{64}{9} \right] $ & X & X & X \\ \hline
$(2,1)$ & $\left[ \nicefrac{4}{9} , \nicefrac{32}{9} \right] \times \left[ \nicefrac{2}{9} , \nicefrac{16}{9} \right] $ & & &X\\ \hline
$(2,2)$ & $\left[ \nicefrac{4}{9} , \nicefrac{32}{9} \right] \times \left[ \nicefrac{4}{9} , \nicefrac{32}{9} \right] $ & X &&X \\ \hline
$(2,3)$ & $\left[ \nicefrac{4}{9} , \nicefrac{32}{9} \right] \times \left[ \nicefrac{6}{9} , \nicefrac{48}{9} \right] $ & X & X &X\\ \hline
$(2,4)$ & $\left[ \nicefrac{4}{9} , \nicefrac{32}{9} \right] \times \left[ \nicefrac{8}{9} , \nicefrac{64}{9} \right] $ & X & X&X\\ \hline
$(3,1)$ & $\left[ \nicefrac{6}{9} , \nicefrac{48}{9} \right] \times \left[ \nicefrac{2}{9} , \nicefrac{16}{9} \right] $ & & &X\\ \hline
$(3,2)$ & $\left[ \nicefrac{6}{9} , \nicefrac{48}{9} \right] \times \left[ \nicefrac{4}{9} , \nicefrac{32}{9} \right] $ & X & &X\\ \hline
$(3,4)$ & $\left[ \nicefrac{6}{9} , \nicefrac{48}{9} \right] \times \left[ \nicefrac{8}{9} , \nicefrac{64}{9} \right] $ & X & X&X\\ \hline
\end{tabular}
\caption{Relative Hausdorff  boxes around points relevant to maximum $\RH{F,G}$.}\label{T:boxes}
\end{table}
From Table \ref{T:boxes} it is clear that we can restrict ourselves to the case where one of $F$ or $G$ passes through one of $(2,1)$, $(2,2)$, $(3,1)$ or $(3,2)$.  Furthermore, since these boxes all contain the line segment $[1,3]\times \set{1}$ we have that the other graph must have at least two vertices of degree 3.  Thus we may assume without loss of generality that the ccdh for $F$ passes through one $S = \set{(2,1),(2,2),(3,1),(3,2)}$ and the ccdh for $G$ is either $(4,4,2)$ or $(4,4,4)$.  In particular, the ccdh for $G$ at $\frac{32}{9}$ either passes through $\nicefrac{8}{9}$ or $\nicefrac{16}{9}$, respectively.  As these points are in $\RHbox{s}{\epsilon_4}$ for all $s \in S$, this completes the proof.
\end{proof}

We note the upper bound given by Theorem \ref{fig:maxPoss} allows one to define an RH {\it similarity} measure, in the spirit of cosine similarity, with values always in the interval $[0,1]$. We analyze how such a similarity measure might transform the RH distance values of the SNAP data by comparing each RH value with the maximum possible RH distance. 
That is, for each pair $G,H$ of SNAP graphs, we consider the ratio
\[
\frac{\RH{G,H}}{\frac{m^2}{2m+1}-1},
\]
where $m=\max({|G|,|H|})$. In Figure \ref{ratio_plot}, we plot the mean of this ratio for each value of $m$. We observe RH distance values constitute a decreasing proportion of the maximum possible RH distance. Accordingly, in practice, normalizing RH distance values in this way will likely yield values that are vanishingly small in the size of the larger graph being compared. 
Hence, practitioners may choose to further transform such RH similarity values or simply opt to use (unnormalized) smooth RH distance instead.

\subsection{Density}
When quantifying the similarity between discrete objects like graphs, it is important to understand the breadth of possible values that can be assumed. In the case of the RH distance between graphs, such questions are more meaningful when considered asymptotically, since it is trivially true that for any fixed $n$, there are a finite number of $n$-vertex graphs and hence a limited number of possible RH distance values. Below, we show that RH distance is dense in the real numbers, in the following sense:

\begin{theorem}
Let $c \in (0,\frac{1}{2}]$.  There exists a sequence of pairs of
graphs $\set{(F_{n_i},G_{n_i})}_i $, each on $n_i$ vertices, such that \[\lim_{i \rightarrow
  \infty} \frac{\RH{F_{n_i},G_{n_i}}}{n_i} = c.\]
\end{theorem}

\begin{proof}
First we note that $\frac{1}{c}$ can be expressed as $k + \delta$
where $k$ is a positive integer that is at least 2 and $\delta \in [0,1)$. Now let $\set{n_i}$ be the sequence of odd positive integers such that $2k+4 \leq n_1$ and $\ceil{\frac{\delta n_i}{k + \delta}} = \ceil{c\delta n_i}$ is even.  We note that this sequence is infinite and has a bounded gaps between successive entries.  
In particular, for every (sufficiently large) positive integer $t$ there exists an odd $n$ such that $\ceil{\frac{\delta n}{k+\delta}} = 2t$.  To see this, we note that it suffices to show that there is an odd integer such that $\paren{2t-1} <
  \frac{\delta}{k+\delta} n \leq 2t$.  Rearranging this expression, we get
\[2t\paren{1+\frac{k}{\delta}} - \paren{1 + \frac{k}{\delta}} < n \leq
  2t\paren{1+\frac{k}{\delta}}.\] As the gap between the upper and lower bounds on $n$  
is $1 + \frac{k}{\delta} > 3$, there is at least one odd integer for every
choice of $t$ and the gap between successive such odd integers is
bounded by a function of $k$ and $\delta$.  

We now define $G_{n_i}$ as a $r_i$-regular graph on $n_i$ vertices, where $r_i = \floor{\frac{k}{k+\delta}n_i} - 1$.  Note that 
$n_i - r_i = n_i - \floor{\frac{k}{k+\delta}n_i}+1 = \ceil{\frac{\delta}{k+\delta}n_i}  +1$, which is odd.  Thus $n_i$ and $r_i$ have opposite parities, and so $r_in_i$ is even and such a graph exists.  We form $F_{n_i}$ by creating a $(k-2)$-regular graph on $n_i -1$ vertices and adding a dominating vertex.   Note that the ccdh of $G_{n_i}$ is $n_i$ repeated $r_i$
times while the ccdh of $F_{n_i}$ is $n_i$ repeated $k-1$ times followed
by $1$, $n_i-k$ times.

We will show that $\RH{F_{n_i},G_{n_i}} = \tau$ where $\tau = \frac{\paren{r_i - k + 1}n_i -1}{kn_i + 1}$.  We first note that 
\[ r_i \geq \floor{\frac{k(2k+4)}{k+\delta}} - 1 \geq \floor{\frac{2k^2+4k}{k+1}} - 1 = 2k \] 
and thus the ccdhs $G_{n_i}$ and $F_{n_i}$ coincide on the interval $[1,k-1]$.  Thus by Lemma \ref{L:flat}, in order to show that $\RH{F_{n_i},G_{n_i}} \leq \tau$, it suffices to show that $F_{n_i}$ intersects $\RHbox{(r_i,n_i)}{\tau}$ and that $G_{n_i}$ intersects both $\RHbox{(k,1)}{\tau}$ and $\RHbox{(n_i-1,1)}{\tau}$.  Now as 
$r_i -k +1 \geq 2k -k +1 = k+1$, we have that $\tau > 1$.  Thus, trivially, $F_{n_i}$ intersects $\RHbox{(r_i,n_i)}{\tau}$ and $G_{n_i}$ intersects $\RHbox{(n_i-1,1)}{\tau}$.   Now consider
$\RHbox{(k,1)}{\tau}$ and
  note that its upper right coordinate is \[\paren{
     \frac{\paren{r_i+1}n_i}{kn_i+1}k,
     \frac{\paren{r_i+1}n_i}{kn_i+1}}.\] 
Since \[ r_i <  \frac{\paren{r_i+1}n_i}{kn_i+1}k < r_i+1,\] and 
\[ -n_i\paren{ \frac{\paren{r_i+1}n_i}{kn_i+1}k - r_i} + n_i =  -n_i\frac{kn_i - r_i}{kn_i+1} + n_i  = \frac{ -kn_i^2 + r_in_i + kn_i^2 + n_i}{kn_i+1} = \frac{\paren{r_i+1}{n_i}}{kn_i+1},\] we have that $\RHbox{(k,1)}{\tau}$ touches $G_{n_i}$ on the line between $(r_i,n_i)$ and $(r_i+1,0)$.  Further, since $\frac{\paren{r_i+1}n_i}{kn_i+1} < n_i$, this implies that $\RH{F_{n_i},G_{n_i}} = \tau$.   

Finally, noting that 
\[ \lim_{i \rightarrow \infty} \frac{1}{n_i}
  \frac{(r_i-k+1)n_i-1}{kn_i+1} = \lim_{i \rightarrow \infty} \frac{
    \paren{\floor{\frac{k}{k+\delta}n_i} - k}n_i - 1}{n_i(kn_i+1)} =
  \frac{1}{k+\delta} = c,\]
completes the proof.

\end{proof}

\subsection{Extremal edge deletion}
In this section, we aim to uncover which minimal change in graph structure results in a maximal change in RH distance. On this topic, recall that in \cite{Matulef2017, Stolman2017} Matulef and Stolman give an example of two $n$-vertex graphs differing only by a single edge, yet with non-vanishing discrete RH distance and vanishing smooth RH distance, in order to to argue smooth RH distance fixes flaws in the discrete RH distance. Here, we investigate to what extent this property of (smooth) RH distance holds, not only for their particular example, but in general. Namely, we ask: over all possible pairs of $n$-vertex graphs that differ by an edge, what is the maximum (smooth) RH distance, and which pair of graphs achieves this maximum? 

As we soon will show, an RH extremal pair of edge-differing graphs is given by the star graph\footnote{Recall that a star graph on $n$ vertices is the complete bipartite graph, $K_{1,n-1}$.} and a ``perturbed" star graph containing an additional edge between pendant (i.e. degree 1) vertices. More generally, the extremal pair is given precisely by any such ``star-degreed" graph, i.e.:

\begin{definition}
An $n$-vertex graph with precisely one degree $k\geq 2$ vertex and $n-1$ pendant vertices is called {star-degreed}.
\end{definition}
Besides the star graph itself (on at least 3 vertices), disconnected graphs consisting of the disjoint union of a star on $k\geq 3$ vertices and $\nicefrac{n-k}{2}$ many disjoint edges are also star-degreed. Below, we calculate the discrete and smooth RH distance between a star-degreed graph $S$ and graph $S^*=S\cup e$ where $e$ is any edge between pendant vertices in $S$.

\begin{claim}[Star-degreed vs perturbed star-degreed]
For $n \geq 5$, let $S$ be a star-degreed graph, and let $S^*=S \cup e$ where $e$ is any edge between pendant vertices in $S$. 
Then 
\begin{enumerate}
\item[$(i)$] $\RHrd{S,S^*}=1/2$ and $\RHrd{S^*,S}=2/3$.
\item[$(ii)$] $\RHr{S,S^*}=2/5$ and $\RHr{S^*,S}=\frac{2}{2n+1}$.
\end{enumerate}
\end{claim}


\begin{proof}
Letting $F$ and $G$ denote the ccdhs of $S$ and $S^*$, respectively, it is clear that $F(2)=1$, $G(2)=3$, and $F(d)=G(d)$ for all other $d \in \mathbb{Z}_{\geq 1}$. We first prove $(i)$. Fixing $d=2$, we note that $\epsilon^*=\frac{1}{2}$ is the minimum $\epsilon$ such that the set of integers $\{\lceil(1-\epsilon)d\rceil,\dots,\lfloor(1+\epsilon)d\rfloor \} \not= \{2\}$. Then since $\{\lceil(1-\epsilon^*)d\rceil,\dots,\lfloor(1+\epsilon^*)d\rfloor \}=\{1,2,3\}$ and 
\[
\min_{d' \in \{1,2,3\}} \frac{|F(2)-G(d')|}{F(2)} = 0  < \frac{|F(2)-G(2)|}{F(2)}=2,
\]
we have that $\RHrd{S,S^*}=1/2$. Now, switching the roles of $F$ and $G$, we observe $|G(2)-F(d')|/G(2)\geq \frac{2}{3}$ for all $d'$ when $n\geq 5$, with equality holding if $d'=2$
; hence $\RHrd{S^*,S}=2/3$.

Next, we prove $(ii)$. Since for $d' \in [1,2]$, we have
\[
\frac{|F(2)-G(d')|}{F(2)} \geq \frac{|F(2)-G(2)|}{F(2)},
\]
it suffices to restrict attention to $d' \in [2,3]$. Rearranging $\frac{|F(2)-G((1+\epsilon)2)|}{F(2)}\leq \epsilon$ yields $\epsilon\geq 2/5$; hence $\RHr{S,S^*}=2/5$. Now, switching the roles of $F$ and $G$, we observe $|G(2)-F(d')|/G(2)=2/3$ for all $d'\geq 2$. For $d' \in [1,2]$, rearranging $\frac{|G(2)-F((1-\epsilon)2)|}{G(2)}\leq \epsilon$ yields $\epsilon\geq \frac{2}{2n+1}$; hence $\RHr{S^*,S}=\frac{2}{2n+1}$.
\end{proof}

Below, we prove that the star-degreed and perturbed star-degreed graphs considered above form the extremal pair maximizing RH distance over all $n$-vertex graphs differing by an edge. 

\begin{theorem} \label{thm:ext}
Given $G=(V,E)$ and $e \notin E$, let $G \cup e$ denote the graph on $V$ with $E=E(G) \cup e$. Then
\[
\max_{G} \RH{G,G\cup e} = 2/5,
\]
which is achieved by taking $G$ to be a star-degreed graph and $e$ as any edge between pendant vertices in $G$. 
\end{theorem}

\begin{proof}
Let $G= (V,E)$ be a graph of order $n$ and $e = xy \notin E$. Let $F$ denote the ccdh of $G \cup e$. Then 
\[
F(d) = G(d)+\delta_{d, d(x)+1} +\delta_{d,d(y)+1},
\]
where $\delta_{i,j}$ denotes the Kronecker delta function. Thus, we are left to consider when $d$ equals either $d(x)+1$ or $d(y)+1$. Without loss of generality, let $d = d(x)+1$. 
Next note that for any graph $H$, $H(1) = |V(H)|$ and $H(k+1) = H(k)-n_k$, for $k \in \mathbb{Z}_{\geq 1}$, where $n_k$ is the number of vertices in $H$ with degree $k$. From this observation, it follows that  $G(d-1) = G(d)+n_{d-1} \geq F(d)$. Thus, since $G(d-1) \geq F(d) > G(d) \geq F(d)-2$ there exists $d'$ such that $d-d' \leq \frac{4}{5} \leq \frac{2}{5}d$ and $G(d') = F(d) - \frac{2}{5}$, and so $\RHr{F,G} \leq \frac{2}{5}$. Now from the observation above, it also follows that $F(d+1) \leq G(d)$. Thus, since $F(d+1) \leq G(d) < F(d) \leq G(d)+2$, there exists $d'$ such that $d'-d \leq \frac{2}{5}d$ and $F(d') = G(d) +\frac{2}{5}$, and so $\RHr{G,F} \leq \frac{2}{5}$, as well. Therefore $\RH{G,F} \leq \frac{2}{5}$.   Now suppose that $\RH{G,F} = \frac{2}{5}$. Then from our calculations above, we observe that $d = 2$ and $G(d) =1$ so $G$ must be star-degreed. Further, in this case when, for all $d'$ such that $|d-d'| \leq \frac{2}{5}d$, $|G(d)-F(d')| \geq \frac{2}{5}G(d)$, and so
\[
\max_{G} \RH{G,G\cup e} = 2/5,
\]
with equality achieved precisely when $G$ has no isolated vertices and one non-pendant vertex.
\end{proof}

\begin{figure}
\begin{subfigure}[b]{0.5\textwidth}
\[
\includegraphics[width=0.75\linewidth]{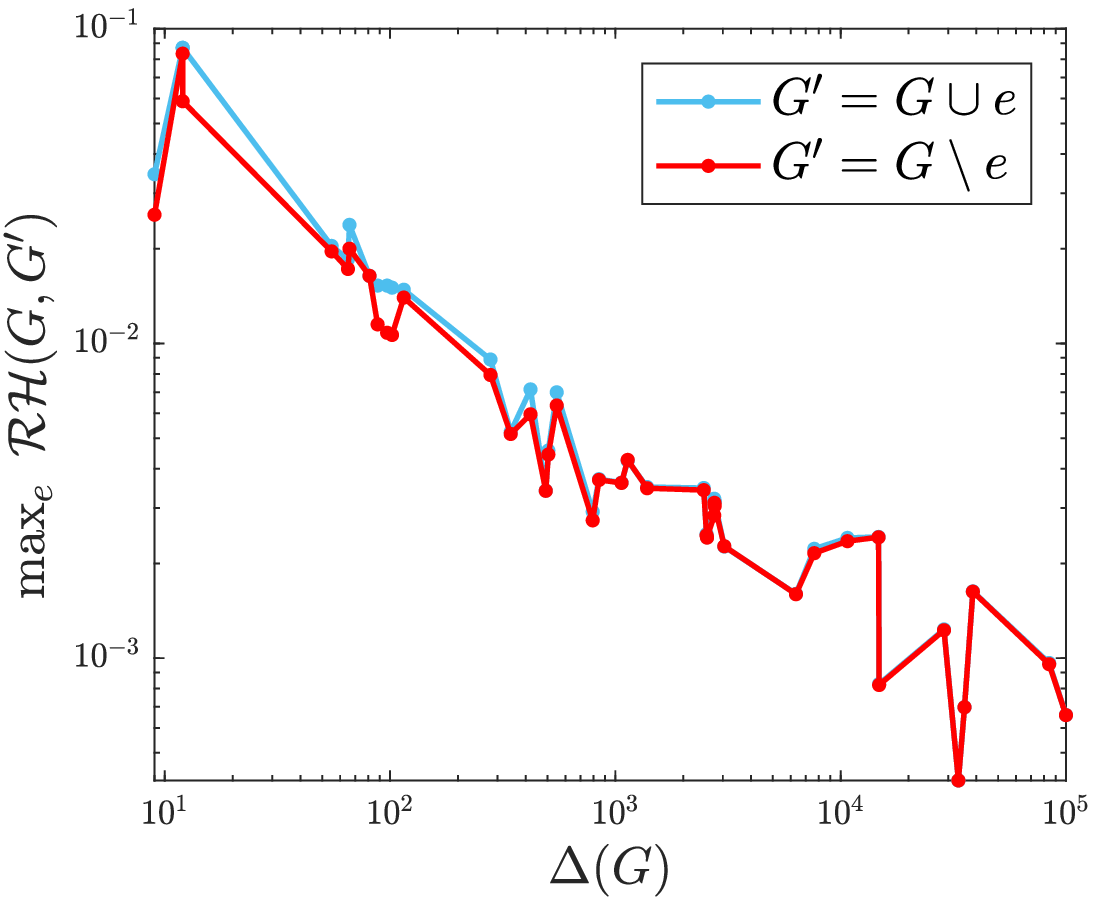}
\]
\caption{}\label{fig:edge_max}
\end{subfigure}
\begin{subfigure}[b]{0.5\textwidth}
\[
\includegraphics[width=0.75\linewidth]{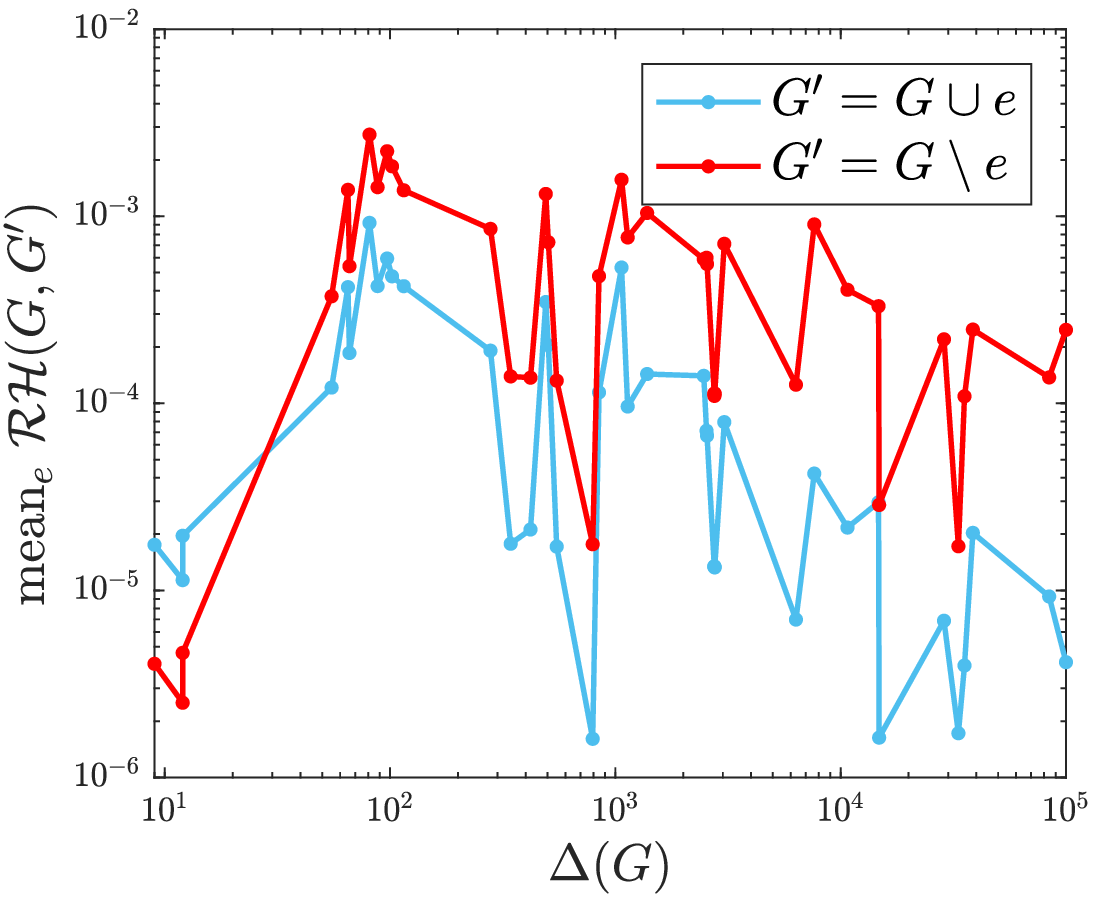}
\]
\caption{}\label{fig:edge_avg}
\end{subfigure}
\caption{The maximum (\ref{fig:edge_max}) and average (\ref{fig:edge_avg}) RH distance produced by single edge additions and deletions on the SNAP graphs. }
\end{figure}

This result shows that RH distance between graphs that differ by a single edge can be nonvanishing as the maximum of the degree tends towards infinity. However, the extremal family of graphs given by Theorem \ref{thm:ext} (i.e. star-degreed graphs) is clearly highly-structured. To investigate whether edge-perturbation can produce unusually large RH distance in practice, we again turn to the SNAP data. For each graph, we compute the maximum RH distance over all possible single edge deletions and additions; that is,
\[
\max_{e\in E(\overline{G})} \RH{G,G\cup e} \mbox{  and  } \max_{e \in E(G)} \RH{G,G\setminus e},
\]
where $\overline{G}$ denotes the complement of $G$. In addition to the maximum, we also computed the average RH distance over edge additions and deletions. Figure \ref{fig:edge_max} shows the maximum RH distance between real graphs differing by an edge decreases in maximum degree of the graph. Figure \ref{fig:edge_avg} presents the average RH distance of single edge perturbations. While the same decreasing trend observed for the maximum is not apparent for the average, the results affirm that the expected RH distance produced by a random edge deletion or addition is small in practice. Both of these results suggest the behavior exhibited by star-degreed graphs is unlikely to be encountered in real data.

\section{Conclusion and future work} \label{sec:conc}

We conducted an algorithmic and analytic study of graph Relative Hausdorff distance, a newly proposed measure for quantifying similarity between degree distributions. Improving upon quadratic algorithms, we devised a linear time algorithm for computing RH distance. Furthermore, we analyzed the basic properties of RH distance, answering a number of natural questions concerning RH distance's status as a bona-fide metric, its range and density of possible values, and its extremal behavior with regard to single edge deletions. While RH distance has so far been mainly applied in the context of evaluating streaming and sampling algorithms \cite{Eden2018, Matulef2017, Simpson2015, Stolman2017}, it is by no means limited to these applications. Indeed, the work in \cite{rhCyber} shows RH distance can be used as anomaly detection method in time-evolving graphs, and is sometimes more sensitive to anomalies than the computationally-expensive edit distance. Nonetheless, more work is needed to determine how RH distance's utility in applications; in this regard, additional comparative studies of RH distance against more established metrics whose behavior is better understood would be valuable. 
Last but not certainly least, we have not explored RH distance in stochastic settings. What is the expected RH distance between two instantiations of a random graph model, or between the expected degree sequence and that of a realization of the random graph model? Such results would not only be of interest in themselves, but provide a new way of quantifying the precision and accuracy of generative graph models in preserving degree distribution.


\section*{Acknowledgments}
We thank Emilie Purvine and Paul Bruillard for their helpful discussions and Carlos Ortiz-Marrero for his assistance with processing the SNAP datasets.  We also thank the anonymous referees whose comments improved the overall quality of this paper.

\bibliographystyle{siamplain}
\bibliography{RH_refs}

\end{document}